\documentclass[10pt]{article}
 \usepackage{amsmath,amsfonts,amssymb,amsthm}
 \usepackage{hyperref}
 \usepackage{multirow}

 \theoremstyle{plain}
 \newtheorem{theorem}{Theorem}[section]
 \newtheorem{lemma}[theorem]{Lemma}
 \newtheorem{corollary}[theorem]{Corollary}

 \theoremstyle{definition}
 \newtheorem{definition}[theorem]{Definition}
 
 \newtheorem{example}[theorem]{Example}

 \theoremstyle{remark}
 \newtheorem*{remark}{Remark}

 \title{Gravitational instantons with faster than quadratic curvature decay (II)}

 \author{Gao Chen and Xiuxiong Chen}

\begin{document}

 \maketitle

\tableofcontents

 \section{Introduction}
  This is our second paper in a series to study gravitational instantons, i.e. complete hyperk\"ahler manifolds of real dimension 4 with faster than quadratic curvature decay, i.e. $|\mathrm{sec}|=O(r^{-2-\epsilon})$, where $\epsilon$ is any small positive number. In our first paper \cite{FirstPaper}, we constructed several standard models $(E,h)$ as possible ends and proved that any gravitational instanton $M$ must be asymptotic to one of the standard models with error $O'(r^{-\epsilon})$.  According to the different choice of $E, $ the gravitational instanton is called ALE, ALF-$A_k$, ALF-$D_k$, ALG, ALH-splitting or ALH-non-splitting. In this paper, we will first improve the asymptotic rate.

  \begin{theorem}
  (Main Theorem 1)
  Given any gravitational instanton $(M,g)$, there exist a bounded domain $K \subset M$ and a diffeomorphism $\Phi:  E\rightarrow M \setminus K$ such that the error term $Err=\Phi^* g-h$ satisfies

  (ALE) $|\nabla^m Err|=O(r^{-4-m}),\forall m\ge 0$.

  (ALF-$A_k$ and ALF-$D_k$) $|\nabla^m Err|=O(r^{-3-m}),\forall m\ge 0$;

  (ALG) $|\nabla^m Err|=O(r^{-\delta-m}),\forall m\ge 0$, where $\delta=\min_{n\in\mathbb{Z},n<2\beta}\frac{2\beta-n}{\beta}.$ In other words,

  \begin{center}
  \begin{tabular}{|c|c|c|c|c|c|c|c|c|}
  \hline
  Type & Regular & I$_0^*$ & II & II$^*$ & III & III$^*$ & IV & IV$^*$\\

  $\beta$ & 1 & $\frac{1}{2}$ & $\frac{1}{6}$ & $\frac{5}{6}$ & $\frac{1}{4}$ & $\frac{3}{4}$ & $\frac{1}{3}$ & $\frac{2}{3}$\\

  $\delta$ & 1 & 2 & 2 & $\frac{4}{5}$ & 2 & $\frac{2}{3}$ & 2 & $\frac{1}{2}$\\
  \hline
  \end{tabular}
  \end{center}

  (ALH-splitting) $Err=0$;

  (ALH-non-splitting) $|\nabla^m Err|=O(e^{-\delta r}),\forall m\ge 0$, where $\delta=2\pi\min_{\lambda\in \Lambda^*\setminus\{0\}}|\lambda|$;
  \end{theorem}

  Actually, we will show that the deformation space of hyperk\"ahler 4-manifolds is a subspace of the space of closed anti-self-dual forms. Therefore, the asymptotic rate is at least the decay rate of the first closed anti-self-dual form.

  Remark that the ALE part of Main Theorem 1 was done by Bando, Kasue and Nakajima \cite{BandoKasueNakajima}. The ALF-$A_k$ part was done by Minerbe \cite{MinerbeMultiTaubNUT}. The ALH-splitting part can be done by the splitting theorem \cite{CheegerGromoll} \cite{FirstPaper}. So we will focus on the other three parts in this paper.

  With the improved asymptotic rate,  we can prove that any ALF gravitational instanton can be compactified in the complex analytic sense. This confirms a conjecture of Yau \cite{Yau} in ALF case.  Following  Kodaira's work \cite{Kodaira},  we can then analyze the topology of the compactification.  This allows us to give a complete classification of ALF-$D_k$ gravitational instantons.

  \begin{theorem}
  (Main Theorem 2) Any ALF-$D_k$ gravitational instanton must be the Cherkis-Hitchin-Ivanov-Kapustin-Lindstr\"om-Ro\v{c}ek metric.
  \end{theorem}

  We will give a precise definition of the Cherkis-Hitchin-Ivanov-Kapustin-Lindstr\"om-Ro\v{c}ek metric as Example \ref{Cherkis-Hitchin-Ivanov-Kapustin-Lindstrom-Rocek} in Section 2.

  To illustrate our method of proving Main Theorem 2, we will first use the same technique to give a new proof of a theorem of Minerbe \cite{MinerbeMultiTaubNUT}:

  \begin{theorem}
  (Minerbe \cite{MinerbeMultiTaubNUT}) Any ALF-$A_k$ gravitational instanton must be the multi-Taub-NUT metric.
  \end{theorem}

  We will give a precise definition of the multi-Taub-NUT metric as Example \ref{multi-Taub-NUT} in Section 2.

  Even though Theorem 1.3 has been proved by Minerbe using other methods, our new proof is meaningful because it's a simplification of Main Theorem 2.

  As a corollary, we will prove a Torelli-type theorem for ALF gravitational instantons as an analogy of Kronheimer's results \cite{Kronheimer1} \cite{Kronheimer2}:

  \begin{corollary}
  (Torelli-type theorem for ALF gravitational instantons)

  Let $M$ be the 4-manifold which underlies an ALF-$A_k$ or ALF-$D_k$ gravitational instanton. Let $[\alpha^1],[\alpha^2],[\alpha^3]\in H^2(M,\mathbb{R})$ be three cohomology classes. Let $L>0$ be any positive number. Then there exists on $M$ an ALF hyperk\"ahler structure for which the cohomology classes of the K\"ahler forms $[\omega^i]$ are the given $[\alpha^i]$ and the length of the asymptotic $S^1$-fiber goes to $L$ at infinity. It's unique up to isometries which respect $I$,$J$, and $K$. Moreover, it's non-singular if and only if $[\alpha^i]$ satisfy the nondegeneracy condition:

  For each $[\Sigma]\in H_2(M,\mathbb{Z})$ with $[\Sigma]^2=-2$, there exists $i\in\{1,2,3\}$ with $[\alpha^i][\Sigma]\not=0$.
  \end{corollary}

  In Section 2, we give several definitions. In Section 3, we develop the deformation theory of hyperk\"ahler 4-manifolds. In Section 4, we use this deformation theory to prove Main Theorem 1. In Section 5, we give a new proof of the classification of ALF-$A_k$ gravitational instantons. In Section 6, we prove Main Theorem 2. In Section 7, we prove the Torelli-type theorem for ALF gravitational instantons as a corollary. In Section 8, we discuss the applications of our work.

  \section{Definitions and Notations}
  In this paper, we follow the notations of our first paper. $r$ is still defined by $r(p)=|\pi(p)|$, where $\pi:E\rightarrow C(X)\setminus B_R$ is the torus fiberation. $O'(r^\alpha)$ still means that for any $m\ge 0$, the $m$-th derivatives of the tensor belong to $O(r^{\alpha-m})$. $\chi$ is a smooth cut-off function from $(-\infty,+\infty)$ to $[0,1]$ such that $\chi\equiv 1$ on $(-\infty,1]$ and $\chi\equiv 0$ on $[2,\infty)$. $B_{\kappa r_i}(p_i)$ still means the cover satisfying very good conditions. $\omega^1$, $\omega^2$, $\omega^3$ are the K\"ahler forms according to $I$, $J$, $K$. $\omega^+=\omega^2+i\omega^3$ is the $I$-holomorphic symplectic form.

  We still use the following definition of twistor space of hyperk\"ahler manifolds.

  \begin{definition}
  (c.f. \cite{HithcinKarlhedLindstromRocek}) Let $(M,g,I,J,K)$ be a hyperk\"ahler manifold. Then the twistor space $Z$ of $M$ is the product manifold $M\times \mathbb{S}^2$ equipped with an integrable complex structure
  $$\underline{I}=(\frac{1-\zeta\bar\zeta}{1+\zeta\bar\zeta}I-\frac{\zeta+\bar\zeta}{1+\zeta\bar\zeta}J
  +i\frac{\zeta-\bar\zeta}{1+\zeta\bar\zeta}K,I_0),$$
  where $\zeta\in\mathbb{C}\subset\mathbb{C}\cup\{\infty\}=\mathbb{CP}^1=\mathbb{S}^2$ is the coordinate function, and $I_0$ is the standard complex structure on $\mathbb{CP}^1$.
  \end{definition}

  Notice that our definition is different from \cite{HithcinKarlhedLindstromRocek} to correct a sign error.

  We can define a form $\omega$ on the twistor space by
  $$\omega=(\omega^2+i\omega^3)+2\zeta\omega^1-\zeta^2(\omega^2-i\omega^3).$$
  It's a holomorphic section of the vector bundle $\Lambda^2T_F^*\otimes \mathcal{O}(2)$, where $F$ means the fiber of $Z$ which is diffeomorphic to $M$.
  We also have a real structure $\tau(p,\zeta)=(p,-1/\bar\zeta)$. It takes the complex structure $\underline{I}$ to its conjugate $-\underline{I}$. In \cite{HithcinKarlhedLindstromRocek}, they proved that $\omega$ and $\tau$ determine the hyperk\"ahler metric.

  We still use the following Hilbert spaces
  \begin{definition}
  Define the ${L^2_{\delta}}$-norm of a tensor by
  $$||\phi||_{L^2_{\delta}(\Omega)}=\sqrt{\int_\Omega|\phi|^2r^\delta\mathrm{dVol}}.$$
  Let $L^2_{\delta}$ be the space of tensors with finite $L^2_{\delta}$-norm.
  Define $\nabla\phi=\psi$ in the distribution sense if for any $\xi\in\mathrm{C}^{\infty}_0$, $(\phi,\nabla^*\xi)=(\psi,\xi)$. Let $H^2_{\delta}$ be the space of  all tensors $\phi$
  such that
  $$\phi \in L^2_{\delta},  \qquad  \nabla \phi \in L^2_{\delta+2}\qquad {\rm and }\qquad   \nabla^2 \phi \in L^2_{\delta+4}.$$
  We can define the norm in this weighted space by
  $$||\phi||_{H^2_{\delta}(\Omega)}=\sqrt{\int_\Omega|\phi|^2r^\delta\mathrm{dVol}+\int_\Omega|\nabla\phi|^2r^{\delta+2}\mathrm{dVol}
  +\int_\Omega|\nabla^2\phi|^2r^{\delta+4}\mathrm{dVol}}.$$
  The inner product is defined accordingly.
  \end{definition}

  In our first paper, we proved that compact supported smooth tensors are dense in those Hilbert spaces. Therefore, it's enough to prove something for compact supported smooth tensors. In ALH-non-splitting case, we can define similar Hilbert spaces using exponential growth weights.

  Now let's define the multi-Taub-NUT metric and more general Gibbons-Hawking ansatz:

  \begin{example}
  \label{multi-Taub-NUT}
  Let $$V(\mathbf{x})=1+\sum_{\alpha=1}^{k+1}\frac{2m}{|\mathbf{x}-\mathbf{x}_{\alpha}|}.$$
  Let $\pi:M_0\rightarrow\mathbb{R}^3\setminus\{\mathbf{x}_{\alpha}\}$ be the $\mathbb{S}^1$-bundle of Euler class -1 around each $\mathbf{x}_{\alpha}$. Let $\eta$ be the connection form with curvature $\mathrm{d}\eta=*\mathrm{d}V$.
  Then $$g=V\mathrm{d}\mathbf{x}^2+V^{-1}\eta^2$$ gives a metric on $M_0$. Let $M=M_0\cup\{p_\alpha\}$ be the completion. Then $M$ is called the multi-Taub-NUT metric with total mass $(k+1)m$. When $k=-1$, $M$ is the trivial product of $\mathbb{S}^1$ and $\mathbb{R}^3$.

  More generally, as long as $V$ is harmonic, we can do the similar construction and call $M$ the Gibbons-Hawking ansatz.
  It has complex structures satisfying
  $$\mathrm{d}x^1=I^*(V^{-1}\eta)=J^*\mathrm{d}x^2=K^*\mathrm{d}x^3.$$
  \end{example}

  Now let's recall the holomorphic structure of the multi-Taub-NUT metric proved by Claude LeBrun \cite{LeBrun}

  \begin{theorem}
  (LeBrun)$(M,I)$ is biholomorphic to the manifold $$uv=\prod_{\alpha=1}^{k+1} (z-(-x^3_\alpha+ix^2_\alpha))$$ if $-x^3_\alpha+ix^2_\alpha$ are distinct or the minimal resolution of it otherwise.
  \label{LeBrun}
  \end{theorem}

  \begin{proof}
  The function $z=-x^3+ix^2$ is an $I$-holomorphic function on $M$. We can define a holomorphic vector field $X$ by
  $$\omega^{+}(X,Y)=-i\mathrm{d}z(Y).$$
  The action of $X$ gives $\mathbb{C}^*$-orbits of $M$. For $z\not=-x^3_\alpha+ix^2_\alpha$, there is only one $\mathbb{C}^*$-orbit. If $-x^3_\alpha+ix^2_\alpha$ are distinct, each $\{z=-x^3_\alpha+ix^2_\alpha\}$ is divided into three $\mathbb{C}^*$-orbits: $\{x^1<x^1_{\alpha}\}$ ,$\{x^1=x^1_{\alpha}\}$, and $\{x^1>x^1_{\alpha}\}$.
  Let $$M^-=\{z\not=-x^3_\alpha+ix^2_\alpha\}\cup\cup_{\alpha}\{z=-x^3_\alpha+ix^2_\alpha,x^1<x^1_{\alpha}\},$$
  $$M^+=\{z\not=-x^3_\alpha+ix^2_\alpha\}\cup\cup_{\alpha}\{z=-x^3_\alpha+ix^2_\alpha,x^1>x^1_{\alpha}\}.$$
  Then both $M^+$ and $M^-$ are biholomorphic to $\mathbb{C}\times\mathbb{C}^*$.
  On the overlap, we have $(z,v)_{M^+}\sim(z,\frac{v}{\prod (z-(-x^3_\alpha+ix^2_\alpha))})_{M^-}$.

  If $-x^3_\alpha+ix^2_\alpha$ are not distinct. For example, suppose we have two points $(0,0,0)$ and $(1,0,0)$. Then we can define $$M_1=\{z\not=0\}\cup\{z=0,x^1<0\},$$
  $$M_2=\{z\not=0\}\cup\{z=0,0<x^1<1\},$$
  $$M_3=\{z\not=0\}\cup\{z=0,x^1>1\}.$$
  On the overlap, we have $(z,v)_{M_1}\sim(z,\frac{v}{z})_{M_2}\sim(z,\frac{v}{z^2})_{M_3}$.
  In other words, it's the minimal resolution of $uv=z^2$ in the sense that we replace the point $\{u=v=z=0\}$ by $\mathbb{CP}^1=\{z=0,0\le x^1\le 1\}.$
  It's similar in general case.
  \end{proof}

  Using Theorem \ref{LeBrun}, we can get the twistor description of the multi-Taub-NUT metric as in \cite{CherkisKapustin}.

  \begin{example}
  Let $U$ be the affine variety in $\mathbb{C}^4$ with coordinates $(\zeta,z,\rho,\xi)$ defined by
  $$\rho\xi=\prod_{\alpha=1}^{k+1}(z-P_{\alpha}(\zeta))$$
  or the minimal resolution of it, where $$P_{\alpha}(\zeta)=a_{\alpha}\zeta^2+2b_{\alpha}\zeta-\bar a_{\alpha}$$ with parameters $a_{\alpha}\in\mathbb{C}$ and $b_{\alpha}\in\mathbb{R}$.
  Take two copies of $U$ and glue them together over $\zeta\not=0,\infty$ by
  $$\tilde\zeta=\zeta^{-1},$$
  $$\tilde z=\zeta^{-2}z,$$
  $$\tilde\rho=e^{-z/\zeta}\zeta^{-k-1}\rho,$$
  $$\tilde\xi=e^{z/\zeta}\zeta^{-k-1}\xi.$$
  Then $\zeta$ lies in $\mathbb{CP}^1=\mathbb{C}\cup\{\infty\}$ and $z$ is a section in $\mathcal{O}(2)$.
  Define $$\omega=4i\mathrm{d}\log\rho\wedge\mathrm{d}z=i\mathrm{d}z\wedge\mathrm{d}\chi.$$
  Define the real structure $\tau$ by
  $$\tau(\zeta,z,\rho,\xi)=(-1/\bar \zeta,-\bar z/\bar\zeta^{2},e^{\bar z/\bar\zeta}(1/\bar \zeta)^{k+1}\bar\xi,e^{-\bar z/\bar\zeta}(-1/\bar \zeta)^{k+1}\bar\rho).$$
  The gluing of $U$ and $\tilde U$ is the twistor space of the multi-Taub-NUT metric up to rescaling.

  Notice that our convention is slightly different from \cite{CherkisKapustin}. We use the real form $i\partial\bar\partial K$ as the K\"ahler form but they use $\partial\bar\partial K$ following the convention of \cite{IvanovRocek}. The other difference is that they use the scaling parameter $\mu$ but we rescale our metric to make $\mu=1$.
  \end{example}

  Similarly, we can define the twistor space of a hyperk\"ahler 8-manifold:

  \begin{example}
  Let $U$ be the subvariety in $\mathbb{C}^7$ with coordinates $(\zeta,w,z,\rho_0,\rho_1,\xi_0,\xi_1)$ defined by
  $$(\rho_0+\rho_1\eta)(\xi_0+\xi_1\eta)=\prod_{\alpha=1}^{k}(\eta-P_{\alpha}(\zeta)) \mod \eta^2-w\eta-z=0,$$
  where $$P_{\alpha}(\zeta)=a_{\alpha}\zeta^2+2b_{\alpha}\zeta-\bar a_{\alpha}$$ with parameters $a_{\alpha}\in\mathbb{C}$ and $b_{\alpha}\in\mathbb{R}$.

  Take two copies of $U$ and glue them together over $\zeta\not=0,\infty$ by
  $$\tilde\zeta=\zeta^{-1},$$
  $$\tilde w=\zeta^{-2}w,\tilde z=\zeta^{-4}z$$
  $$(\tilde\rho_0+\tilde\rho_1\tilde\eta)=e^{-\eta/\zeta}\zeta^{-k}(\rho_0+\rho_1\eta)\mod \eta^2-w\eta-z=0,$$
  $$(\tilde\xi_0+\tilde\xi_1\tilde\eta)=e^{\eta/\zeta}\zeta^{-k}(\xi_0+\xi_1\eta)\mod \eta^2-w\eta-z=0,$$
  Then $\zeta$ lies in $\mathbb{CP}^1=\mathbb{C}\cup\{\infty\}$ and $z$ is a section in $\mathcal{O}(4)$.
  Define $$\omega=4i\sum_{j=1}^{2}\frac{(\mathrm{d}\rho_0+\beta_j\mathrm{d}\rho_1)\wedge\mathrm{d}\beta_j}{\rho_0+\beta_j\rho_1},$$
  where $\beta_1,\beta_2$ are the two roots of $\eta^2-w\eta-z=0$.
  Define the real structure by $$\tau(\zeta,z,\rho,\xi)=(-1/\bar \zeta,\bar z/\bar\zeta^{4},e^{ \bar\eta/\bar\zeta}(1/\bar \zeta)^{k}\bar\xi,e^{- \bar\eta/\bar\zeta}(-1/\bar \zeta)^{k}\bar\rho),$$ where $\rho=\rho_0+\rho_1\eta$ and $\xi=\xi_0+\xi_1\eta$.
  We can realize the gluing of $U$ and $\tilde U$ as the twistor space of a complete hyperk\"ahler 8-manifold.
  \end{example}

  The hyperk\"ahler quotient of the previous example is the Cherkis-Hitchin-Ivanov-Kapustin-Lindstr\"om-Ro\v{c}ek metric. When $k=0$, it's the famous Atiyah-Hitchin metric. Actually, the Atiyah-Hitchin metric \cite{AtiyahHitchin} provided the first example of ALF-$D_k$ gravitational instantons. Later, Ivanov and Ro\v{c}ek \cite{IvanovRocek} conjectured a formula for positive $k$ using generalized Legendre transform developed by Lindstr\"om and Ro\v{c}ek. Cherkis and Kapustin \cite{CherkisKapustin} confirmed this formula. This metric was computed more explicitly by Cherkis and Hitchin \cite{CherkisHitchin}.

  \begin{example}
  \label{Cherkis-Hitchin-Ivanov-Kapustin-Lindstrom-Rocek}
  In the previous example, we look at the $\mathbb{C}^*$ action by $\rho_j\rightarrow \lambda\rho_j$ and $\xi_j\rightarrow \lambda^{-1}\xi_j$. The moment map is $w$. To get the hyperk\"ahler quotient, we set $w=0$ and take the $\mathbb{C}^*$ quotient.

  The submanifold $w=0$ in $U$ can be written as
  $$\begin{array}{lcl} \rho_0\xi_0+z\rho_1\xi_1 & = & p(z),\\
  \rho_1\xi_0+\rho_0\xi_1 & = & q(z),\end{array} $$
  where $$\prod_\alpha(\eta-P_{\alpha})=p(z)+\eta q(z) \mod \eta^2-z=0.$$
  The $\mathbb{C}^*$-quotient can be obtained by using the $\mathbb{C}^*$-invariant coordinates
  $$\begin{array}{lcl} x & = & i^k[\rho_1\xi_0-\rho_0\xi_1],\\
  y & = & i^k[-2\rho_1\xi_1+r(z)],\end{array} $$
  where $$p(z)=zr(z)+\prod_{\alpha}(-P_{\alpha}).$$
  Thus $$\rho_1\xi_0=\frac{q(z)+(-i)^k x}{2},\rho_0\xi_1=\frac{q(z)-(-i)^k x}{2},$$
  $$\rho_1\xi_1=\frac{r(z)-(-i)^k y}{2},\rho_0\xi_0=\frac{zr(z)+(-i)^k zy}{2}+\prod_\alpha(-P_{\alpha}).$$
  The equation $$(\rho_0\xi_0)(\rho_1\xi_1)=(\rho_0\xi_1)(\rho_1\xi_0)$$ is reduced to
  $$x^2-zy^2=\frac{1}{-z}(\prod_{\alpha}(z-P^2_{\alpha})-\prod_{\alpha}(-P^2_{\alpha}))+2\prod_{\alpha}(-iP_{\alpha})y.$$
  Moreover,
  $$\omega=i\mathrm{d}(\frac{1}{\sqrt{z}}\log(\frac{yz+\prod_{\alpha}(-iP_{\alpha})
  +\sqrt{z}x}{yz+\prod_{\alpha}(-iP_{\alpha})-\sqrt{z}x}))\wedge\mathrm{d}z.$$
  This gives the twistor space of the Cherkis-Hitchin-Ivanov-Kapustin-Lindstr\"om-Ro\v{c}ek metric.
  \end{example}

  When $P_{\alpha}(\zeta)$ and $-P_{\alpha}(\zeta)$ are distinct, the manifold is non-singular. Otherwise, the CHIKLR metric is the minimal resolution of the singular manifold which will be discussed later. It's interesting to notice that \cite{ChalmersRocekWiles} when two $P_{\alpha}$ equal to 0, the singular manifold is the $\mathbb{Z}_2$-quotient of the multi-Taub-NUT metric. Moreover, when $k\ge 3$, if all of $P_{\alpha}$ equal to 0, the singular manifold is exactly the quotient of the Taub-NUT metric by the binary dihedral group $D_{4(k-2)}$ because of the following calculation:

  \begin{example}
  It's well known that the Taub-NUT metric is biholomorphic to $\mathbb{C}^2$. Let $u$, $v$ be the coordinates of $\mathbb{C}^2$. Define the action of the binary dihedral group $$D_{4(k-2)} = <\sigma,\tau|\sigma^{2k-4}=1,\sigma^{k-2}=\tau^2,\tau\sigma\tau^{-1}=\sigma^{-1}>$$ by
  $$\tau(u,v)=(v,-u), \sigma(u,v)=(e^{i\pi/(k-2)}u,e^{-i\pi/(k-2)}v).$$
  Then $$x=uv(u^{2k-4}-v^{2k-4})/2, y=(u^{2k-4}+v^{2k-4})/2, z=u^2v^2$$ are invariant under the action with the relationship
  $x^2-zy^2=-z^{k-1}$. This is exactly the previous example with all $P_{\alpha}=0$.
  \end{example}

\section{Deformation of hyperk\"ahler 4-manifolds}

  It's well known that in real dimension 4, the hyperk\"ahler condition is equivalent to the Calabi-Yau condition. So we can study the deformation theory by viewing them as Calabi-Yau manifolds. However, to keep track of the symmetry between three complex structures, we prefer a more direct approach inspired by the lecture of Sir Simon Donaldson in the spring of 2015 at Stony Brook University.

  \begin{lemma}
  A 4-manifold is hyperk\"ahler if and only if there exist three closed 2-forms $\omega^i$ satisfying
  $$\omega^i\wedge\omega^j=2\delta_{ij}V,$$
  where $V$ is a nowhere vanishing 4-form.
  \end{lemma}

  \begin{proof}
  Given three 2-forms, we can call the linear span of them the ``self-dual" space. The orthogonal complement of the ``self-dual" space under wedge product is called ``anti-self-dual" space. These two spaces determine a star operator. It's well known that the star operator determine a conformal class of metric. We can then determine the conformal factor by requiring $V$ to be the volume form. Using this metric and the three forms $\omega^i$, we can determine three almost complex structures $I$, $J$ and $K$. It's easy to see that $IJ=K$ or $IJ=-K$. Since the two cases are disconnected, we can without loss of generality assume that the first case happens. By Lemma 6.8 of \cite{Hitchin}, $I$,$J$,$K$ are parallel.
  \end{proof}

  Therefore, given a family of hyperk\"ahler metrics $\omega^i(t)$ on a fixed manifold $M$, the deformations $\theta^i=\frac{\mathrm{d}}{\mathrm{d}t}\omega^i(t)|_{t=0}$ satisfy
  $$\omega^i\wedge\theta^j+\omega^j\wedge\theta^i=0, i\not=j;$$
  $$\omega^1\wedge\theta^1=\omega^2\wedge\theta^2=\omega^3\wedge\theta^3.$$
  Notice that the anti-self-dual components of $\theta^i$ don't affect the equation, so we only need to look at the self-dual components. Let $$V=\{\mathbf{\theta}\in\Lambda^{+}\oplus\Lambda^{+}\oplus\Lambda^{+}:\omega^i\wedge\theta^j+\omega^j\wedge\theta^i=0, i\not=j;
  \omega^1\wedge\theta^1=\omega^2\wedge\theta^2=\omega^3\wedge\theta^3\}.$$
  Then $V$ is generated by the following basis:
  $$e_1:\theta^1=\omega^1,\theta^2=\omega^2,\theta^3=\omega^3;$$
  $$e_2:\theta^1=0,\theta^2=\omega^3,\theta^3=-\omega^2;$$
  $$e_3:\theta^1=-\omega^3,\theta^2=0,\theta^3=\omega^1;$$
  $$e_4:\theta^1=\omega^2,\theta^2=-\omega^1,\theta^3=0.$$
  Now we look at the action of diffeomorphism group. The infinitesimal diffeomorphism group $X$ acts simply by $\theta^i=L_{X}\omega^i=\mathrm{d}(X\rfloor \omega^i).$
  The projection of $L_{X}\omega^i$ to $V$ defines an operator $\mathcal{D}:\mathrm{Vect}(M)\rightarrow V$.
  Notice that $\mathcal{D}$ is canonically determined by $\omega^i$. In particular, if there is a symmetry group $G$, then $\mathcal{D}$ is also invariant under $G$.

  On $\mathbb{R}^4$,
  $$\omega^1=\mathrm{d}x^1\wedge\mathrm{d}x^2+\mathrm{d}x^3\wedge\mathrm{d}x^4,$$
  $$\omega^2=\mathrm{d}x^1\wedge\mathrm{d}x^3+\mathrm{d}x^4\wedge\mathrm{d}x^2,$$
  $$\omega^3=\mathrm{d}x^1\wedge\mathrm{d}x^4+\mathrm{d}x^2\wedge\mathrm{d}x^3.$$
  It's easy to compute that $$\mathcal{D}(f^1\frac{\partial}{\partial x^1}+f^2\frac{\partial}{\partial x^2}+f^3\frac{\partial}{\partial x^3}+f^4\frac{\partial}{\partial x^4})$$
  $$=(\frac{\partial f^1}{\partial x^1}+\frac{\partial f^2}{\partial x^2}+\frac{\partial f^3}{\partial x^3}+\frac{\partial f^4}{\partial x^4})e_1+(\frac{\partial f^1}{\partial x^2}-\frac{\partial f^2}{\partial x^1}+\frac{\partial f^3}{\partial x^4}-\frac{\partial f^4}{\partial x^3})e_2$$
  $$+(\frac{\partial f^1}{\partial x^3}-\frac{\partial f^2}{\partial x^4}-\frac{\partial f^3}{\partial x^1}+\frac{\partial f^4}{\partial x^2})e_3+(\frac{\partial f^1}{\partial x^4}+\frac{\partial f^2}{\partial x^3}-\frac{\partial f^3}{\partial x^2}-\frac{\partial f^4}{\partial x^1})e_4.$$

  So $\mathcal{D}\mathcal{D}^*=\Delta$ on $\mathbb{R}^4$.

  On the general hyperk\"ahler 4-manifold, suppose that $\mathcal{D}$ has full image, then we can without loss of generally assume that $\theta^i$ are all anti-self-dual. Notice that they must be closed as the variation of closed forms. They must also be co-closed since $\mathrm{d}(*\theta^i)=-\mathrm{d}\theta^i=0$.
  In other words, the deformation space of hyperk\"ahler 4-manifolds is a subspace of three copies of the space of anti-self-dual harmonic 2-forms. There may be further reductions if $(L_X\omega^1,L_X\omega^2,L_X\omega^3)$ is anti-self-dual harmonic for some vector field $X$.

 \section{Asymptotic behavior of gravitational instantons}

  In this section, we use the principles in the previous section to prove Main Theorem 1.

  To prove that $\mathcal{D}$ has full image, we instead prove that $L=\mathcal{D}\mathcal{D}^*$ has full image. Since $L$ is asymptotic to the Laplacian operator, it's enough to apply Lemma 4 of Minerbe's paper \cite{MinerbeMass} in ALF case and its generalization in ALG case.

  To get this generalization, we still use the decomposition of any tensor $\phi$ into $\mathbb{T}^k$-invariant part $\phi_1$ and the other part $\phi_2$ satisfying $\int_{\pi^{-1}(x)}\phi_2=0$. Notice that any $\mathbb{T}^k$-invariant operator $L$ preserves this decomposition.
  Actually, let $\Phi_t(x,\theta)=(x,\theta+t)$ be the diffeomorphism in local coordinates, then $L$ commutes with $\Phi_t^*$. So
  $$\Phi_t^*(L\phi_1)=L(\Phi_t^*\phi_1)=L\phi_1,$$ and
  $$\int_{t\in\mathbb{T}^k}\Phi_t^*(L\phi_2)=L\int_{t\in\mathbb{T}^k}\Phi_t^*\phi_2=0.$$

  Now let's state several lemmas first.
  \begin{lemma}
  For any $\delta\not=1$, there exists a bounded linear operator $$S:L^2_{\delta}([R,\infty))\rightarrow H^{1}_{\delta-2}([R,\infty))$$ with
  $||S||\le \frac{2}{|\delta-1|}+1$ such that $(Sf)'=f$ in the distribution sense.
  \end{lemma}
  \begin{proof}
  Since $(Sf)'=f$, it's enough to control the $L^{2}_{\delta-2}$-norm of $Sf$.
  We can further reduce to prove the same estimate for $f\in C^{\infty}_0$. So we can assume that $\mathrm{supp}(f)\subset[R_1,R_2]$ with $R<R_1<R_2<\infty$.
  If $\delta>1$, define $$Sf(r)=-\int^{\infty}_{r}f(t)\mathrm{d}t.$$
  So $$||Sf||_{L^2_{\delta-2}([R,\infty))}^2
  =\int_{R}^{\infty}[\int_{r}^{\infty}f(t)\mathrm{d}t]^2r^{\delta-2}\mathrm{d}r$$
  $$=\frac{2}{\delta-1}\int_{R}^{\infty}[\int_{r}^{\infty}f(t)\mathrm{d}t]f(r)r^{\delta-1}\mathrm{d}r
  -\frac{1}{\delta-1}R^{\delta-1}[\int_{R}^{\infty}f(t)\mathrm{d}t]^2$$
  $$\le\frac{2}{\delta-1}\sqrt{\int_{R}^{\infty}[\int_{r}^{\infty}f(t)\mathrm{d}t]^2r^{\delta-2}\mathrm{d}r}
  \sqrt{\int_{R}^{\infty}f^2(t)r^{\delta}\mathrm{d}r}-0$$
  by integral by parts and the Cauchy-Schwarz inequality.

  If $\delta<1$, define $$Sf(r)=\int^{r}_{R}f(t)\mathrm{d}t.$$
  Then $Sf$ is constant for $r>R_2$, and therefore belongs to $L^{2}_{\delta-2}$. Moreover, it's 0 for $r<R_1$.
  Therefore, we can apply the proof of Proposition 4.2 and Theorem 4.5 of our first paper \cite{FirstPaper} to get the required estimate.
  \end{proof}

  \begin{lemma}
  Suppose $\phi\in C^{\infty}(\overline{B_{\tilde R}\setminus B_R})$ is a tensor vanishing on the boundary satisfying $\int_{\pi^{-1}(x)}\phi=0$ for any $x\in B_{\tilde R}\setminus B_R$. Suppose $$\mathbf{L}=\mathbf{A}^{ij}\nabla_i\nabla_j+\mathbf{B}^{i}\nabla_i+\mathbf{C}$$ is a $\mathbb{T}^k$-invariant tensor-valued second order elliptic operator with $$|\mathbf{A}^{ij}-\delta^{ij}\mathbf{Id}|\le Cr^{-\epsilon},|\mathbf{B}^{i}|\le Cr^{-1-\epsilon},|\mathbf{C}|\le Cr^{-2-\epsilon}.$$
  Then as long as $R$ is large enough,
  $$\int_{B_{\tilde R}\setminus B_R}|\phi|^2r^{\delta}+\int_{B_{\tilde R}\setminus B_R}|\nabla\phi|^2r^{\delta}+\int_{B_{\tilde R}\setminus B_R}|\nabla^2\phi|^2r^{\delta}\le C\int_{B_{\tilde R}\setminus B_R}|L\phi|^2r^{\delta}$$
  for any $\delta\in\mathbb{R}$, with constant $C$ independent of $R$ and $\tilde R$.
  \label{estimate-of-oscillation-part}
  \end{lemma}
  \begin{proof}
  It's easy to see that
  $$\int_{B_{\tilde R}\setminus B_R}|\nabla^2\phi|^2r^{\delta}+\int_{B_{\tilde R}\setminus B_R}|\nabla\phi|^2r^{\delta}\le C(\int_{B_{\tilde R}\setminus B_R}|\phi|^2r^{\delta}+\int_{B_{\tilde R}\setminus B_R}|L\phi|^2r^{\delta})$$
  by Theorem 9.11 of \cite{GilbargTrudinger}. By our Theorem 4.9 of \cite{FirstPaper}, if $\phi$ is compactly supported in $B_{\kappa r_i(p_i)}$,
  $$\int_{(B_{\tilde R}\setminus B_R)\cap B_{\kappa r_i(p_i)}}|\phi|^2\le C\int_{(B_{\tilde R}\setminus B_R)\cap B_{\kappa r_i(p_i)}}|\nabla\phi|^2\le C\int_{(B_{\tilde R}\setminus B_R)\cap B_{\kappa r_i(p_i)}}|\Delta\phi|^2.$$
  Therefore
  $$\int|\phi|^2r^{\delta}+\int|\nabla\phi|^2r^{\delta}\le C\sum(\int|\chi_i\phi|^2r^{\delta}+\int|\nabla(\chi_i\phi)|^2r^{\delta})$$
  $$\le C\sum\int|\Delta(\chi_i\phi)|^2r^{\delta}\le C\sum(\int\chi_i^2|\Delta\phi|^2r^{\delta}+\int|\nabla\chi_i|^2|\nabla\phi|^2r^{\delta}+\int|\Delta\chi_i|^2|\phi|^2r^{\delta}).$$
  Here, the first inequality holds because $R$ is large enough and by our first paper \cite{FirstPaper}, we can choose the charts properly so that the number of charts overlapping at any given point is uniformly bounded.

  Notice that $\nabla\chi_i=O(r^{-1})$ and $\nabla^2\chi_i=O(r^{-2})$. By canceling terms, we can prove the theorem for $L=\Delta=\mathrm{Tr}\nabla^*\nabla$. By the same reasons, it can be generalized to more general operator $L$ whose coefficients equal to the Laplacian operator plus small error terms.
  \end{proof}

  From the approximation by $\phi_n=\phi\chi(r-n)$, the condition in Lemma \ref{estimate-of-oscillation-part} that $\phi\in C^{\infty}(\overline{B_{\tilde R}\setminus B_R})$ vanishes on the boundary can be replaced by the condition that $\phi\in C^{\infty}(\overline{B_R^c})$ vanishes on $\partial B_R$ and $\phi, \nabla\phi, \nabla^2\phi\in L^2_{\delta}$.

  Notice that the estimate in Lemma \ref{estimate-of-oscillation-part} doesn't scale correctly. Therefore, we must have the following fact:

  \begin{theorem}
  If $\phi\in H^{2}_{\delta}(B_R^{c})$ satisfies $L\phi=0$. Then $\phi$ is $\mathbb{T}^k$-invariant plus exponentially decay term.
  \label{exponential-decay}
  \end{theorem}
  \begin{proof}
  We can assume $\int_{\pi^{-1}(x)}\phi=0$ and prove that $\phi$ decay exponentially. For any $R$ large enough, we can apply Lemma \ref{estimate-of-oscillation-part} to $(1-\chi(r-R))\phi$.
  Therefore, $$\int_{r>R+2}|\phi|^2r^{\delta}\le C\int_{r>R}|L((1-\chi(r-R))\phi)|^2r^{\delta}$$
  $$\le C \int_{R+1<r<R+2}(|\nabla^2\phi|^2+|\nabla\phi|^2+|\phi|^2)r^{\delta}
  \le C \int_{R<r<R+3}|\phi|^2r^{\delta}$$
  for some constant $C$ independent of $R$. The last inequality holds by Theorem 9.11 of \cite{GilbargTrudinger}.
  So $\int_{r>R}|\phi|^2r^{\delta}$ decay exponentially. $\phi$ also decay exponentially in $L^{\infty}$ norm by Theorem 9.20 of \cite{GilbargTrudinger}.
  \end{proof}

  Now we are able to prove the following generalization of Lemma 4 of Minerbe's paper \cite{MinerbeMass}.

  \begin{theorem}
  As long as $30\delta$ is not an integer, there exists a bounded operator $G_L:L^{2}_{\delta}(B_R^c)\rightarrow H^{2}_{\delta-4}(B_R^c)$ such that $L(G_L \phi)=\phi$.
  \label{full-image}
  \end{theorem}
  \begin{proof}
  It's enough to prove the same thing for $L=\Delta=\mathrm{Tr}\nabla^*\nabla$ and for smooth tensor $\phi$.
  For $\mathbb{T}^{k}$-invariant part, we can use the spectral decomposition as in Theorem 4.5 of our first paper \cite{FirstPaper}.
  For the other part, we can solve the equation $\Delta\psi=\phi$ in $B_{\tilde R}\setminus B_R$ and $\psi=0$ on $\partial(B_{\tilde R}\setminus B_R)$. It's solvable because we can solve it in $H^{1}_0$ first, i.e $(\nabla \psi,\nabla \xi)=(\phi,\xi)$. Then Theorem 8.13 of \cite{GilbargTrudinger} implies that $\psi\in C^{\infty}(\overline{B_{\tilde R}\setminus B_R})$ and vanishes on the boundary. After throwing away the $\mathbb{T}^k$-invariant part, we can apply Lemma \ref{estimate-of-oscillation-part}. Now let $\tilde R$ goes to infinity. We can get a sequence of $\psi_{\tilde R}$. A subsequence converges to a function $\psi_\infty$ in $H^1(B_{\tilde R}\setminus B_R)$ for any $\tilde R$ by Rellich lemma and the diagonal argument. $\psi_\infty$ is a generalized solution since we define derivatives in distribution sense. Notice that actually $\psi_\infty, \nabla\psi_\infty, \nabla^2\psi_\infty\in L^2_{\delta}(B_R^c)$. $\psi_\infty$ also lies in $C^{\infty}(\overline{B_R^c})$ and equals to 0 on $\partial B_R$ by Theorem 8.13 of \cite{GilbargTrudinger}. Therefore, we can apply Lemma \ref{estimate-of-oscillation-part} to $\phi_\infty$. In particular, the difference of two $\psi_\infty$ must be 0. In other words, $\psi_\infty$ is independent of the choice of subsequence. We call that $G_L\phi$.
  \end{proof}

  We are ready to prove Main Theorem 1.

  \begin{theorem}
  Any ALF-$D_k$ gravitational instanton $(M,g)$ is asymptotic to the standard model $(E,h)$ of order 3 in the sense of Section 2 of \cite{FirstPaper}.
  \label{asymptotic-rate-ALF-Dk}
  \end{theorem}

  \begin{proof}
  We already proved that $M$ is asymptotic to $(E,h)$ with error $O'(r^{-\epsilon})$. We will improve the decay rate slightly and iterate the improvement. The decay rates are in $L^\infty$ sense. However, they are also in weighted $L^2$ sense after choosing correct weights. To be convenient, we transfer the weighted $L^2$ estimates back into $L^\infty$ estimates using standard elliptic theory. During this process, the weights are usually slightly changed. Therefore, we will choose irrational $\delta_1<\epsilon$ arbitrarily close to $\epsilon$ and irrational $\delta_2<\delta_1$ arbitrary close to $\delta_1$.

  Let $\mathbf{\omega}_g$ be $(\omega_g^1,\omega_g^2,\omega_g^3)$ and $\mathbf{\omega}_h$ accordingly.
  Then $$\mathbf{\omega}_g-\mathbf{\omega}_h=O'(r^{-\epsilon}).$$
  The difference is small. So we can write it as infinitesimal difference plus some quadratic term. In other words, if we use $h$ to distinguish self-dual and anti-self-dual forms, then the self dual part $$\mathbf{\omega}_g^{+}-\mathbf{\omega}_h=\mathbf{\theta}+O'(r^{-2\epsilon}),$$
  where $\mathbf{\theta}=O'(r^{-\epsilon})\in V$.

  The operator $L=\mathcal{D}\mathcal{D}^{*}$ satisfies all the conditions of Theorem \ref{full-image}. So there exists $G_L$ such that $$\mathbf{\theta}=\mathcal{D}\mathcal{D}^{*}G_L\mathbf{\theta}.$$ Let $X=-\mathcal{D}^{*}G_L\mathbf{\theta}$, then $X=O'(r^{1-\delta_1})$.
  Let $\Phi_t=\exp(-tX)$ be the 1-parameter subgroup of diffeomorphisms generated by $X$.
  Then $$\Phi_t^*(L_X\mathbf{\omega}_h)-L_X\mathbf{\omega}_h=O'(r^{-2\delta_1}), \forall t\in[0,1].$$
  Therefore, $$\Phi_1^*\mathbf{\omega}_h-\mathbf{\omega}_h-L_X\mathbf{\omega}_h
  =\int_{t=0}^{1}(\Phi_t^*(L_X\mathbf{\omega}_h)-L_X\mathbf{\omega}_h)\mathrm{d}t=O'(r^{-2\delta_1}).$$
  So $$(\Phi_1^*\mathbf{\omega}_g)^+ -\mathbf{\omega}_g^+-\mathcal{D}X=(\Phi_1^*\mathbf{\omega}_g-\mathbf{\omega}_g-L_X\mathbf{\omega}_h)^+=O'(r^{-2\delta_1})$$
  because $\mathbf{\omega}_g-\mathbf{\omega}_h=O'(r^{-\epsilon})$.
  After replacing $\mathbf{\omega}_g$ by $\Phi_1^*\mathbf{\omega}_g$, we can assume that $$\mathbf{\omega}_g^{+}-\mathbf{\omega}_h=O'(r^{-2\delta_1}).$$
  We also have $\mathbf{\omega}_g^{-}=O'(r^{-\epsilon})$.
  Write it as $\mathbf{\omega}_g^{-}=\phi+\psi$ with $\phi$ $\mathbb{T}^k$-invariant and $\int_{\pi^{-1}(x)}\psi=0$.
  Since $-\mathrm{d}*\mathbf{\omega}_g^{-}=\mathrm{d}\mathbf{\omega}_g^{-}=-\mathrm{d}\mathbf{\omega}_g^{+}=O'(r^{-2\delta_1-1})$,
  $$(\mathrm{d}\mathrm{d}^*+\mathrm{d}^*\mathrm{d})\mathbf{\omega}_g^{-}=O'(r^{-2\delta_1-2}).$$
  In particular, $\tilde\psi=\psi-G_{\mathrm{d}\mathrm{d}^*+\mathrm{d}^*\mathrm{d}}(\mathrm{d}\mathrm{d}^*+\mathrm{d}^*\mathrm{d})\psi
  $ is harmonic and $\psi-\tilde\psi=O'(r^{-2\delta_2})$. By Theorem \ref{exponential-decay}, $\tilde\psi$ decay exponentially. Therefore, $\psi=O'(r^{-2\delta_2})$.

  Now, we write $\phi$ as $\phi=\alpha\wedge\eta-V*_{\mathbb{R}^3}\alpha$ for $\alpha\in\Lambda^1(\mathbb{R}^3)$.
  Then $$\mathrm{d}\phi=\mathrm{d}\alpha\wedge\eta-\alpha\wedge\mathrm{d}\eta
  -\mathrm{d}V\wedge*_{\mathbb{R}^3}\alpha-V\mathrm{d}(*_{\mathbb{R}^3}\alpha)=O'(r^{-2\delta_1-1})$$
  Let $\delta_3=\min\{2\delta_2,\delta_2+1\}$. Then $\mathrm{d}\alpha=O'(r^{-\delta_3-1})$ and $\mathrm{d}(*_{\mathbb{R}^3}\alpha)=O'(r^{-\delta_3-1})$.
  Therefore $\tilde\alpha=\alpha-G_{\mathrm{d}\mathrm{d}^*+\mathrm{d}^*\mathrm{d}}
  (\mathrm{d}\mathrm{d}^*+\mathrm{d}^*\mathrm{d})\alpha$ is a harmonic 1-form on $\mathbb{R}^3$.
  What's more $\alpha-\tilde\alpha=O'(r^{-\delta_4})$ for all irrational $\delta_4<\delta_3$.

  After a spectral decomposition as in Theorem 4.5 of our first paper \cite{FirstPaper}, we know that $\tilde\alpha=O'(r^{-1})$.

  Combining everything together, the decay rate of $\omega_g-\omega_h$ can be improved to $\min\{\delta_4,1\}$ when we start from $\epsilon$, where the irrational number $\delta_4$ can be arbitrarily close to $\min\{2\epsilon,\epsilon+1\}$. After finite times of iterations, the decay rate of $\omega_g-\omega_h$ can be improved to 1. Moreover, the decay rate of $\mathrm{d}\tilde\alpha$ can be arbitrarily close to 3. Notice that the coefficients of $\mathrm{d}x_i$ in $\tilde\alpha$ is even, so $$\tilde\alpha=\frac{a\mathrm{d}x^1+b\mathrm{d}x^2+c\mathrm{d}x^3}{r}+O'(r^{-3})$$
  for some constants $a$, $b$ and $c$. It's easy to deduce that $a=b=c=0$ from the decay rate of $\mathrm{d}\tilde\alpha$. So $\tilde\alpha=O'(r^{-3})$ instead.
  More iterations yield that the asymptotic rate, i.e the decay rate of $\omega_g-\omega_h$ can be improved to 3.
  \end{proof}

  \begin{remark}
  It's known \cite{Sen} that up to some exponentially decay term, the Cherkis-Hitchin-Ivanov-Kapustin-Lindstr\"om-Ro\v{c}ek metric outside a compact set can be written as the $\mathbb{Z}_2$-quotient of a Gibbons-Hawking ansatz whose $V$ can be written as
  $$V=1-\frac{16m}{|\mathbf{x}|}+\sum_{\alpha=1}^{k}
  (\frac{4m}{|\mathbf{x}-\mathbf{x_\alpha}|}+\frac{4m}{|\mathbf{x}+\mathbf{x_\alpha}|})=1+\frac{8m(k-2)}{r}+O'(r^{-3}).$$
  Therefore, our estimate is optimal. In ALF-$A_k$ case, the coefficients of $\mathrm{d}x_i$ in $\tilde\alpha$ are not necessarily even. So the asymptotic rate is only 2 here. Later, we will use this estimate to give a new proof of Theorem 1.3. Notice that the asymptotic rate of the multi-Taub-NUT metric is actually 3.

  \end{remark}

  \begin{theorem}
  Any ALG gravitational instanton $(M,g)$ is asymptotic to the standard model $(E,h)$ of order $\min_{n\in\mathbb{Z},n<2\beta}\frac{2\beta-n}{\beta}$ in the sense of Section 2 of \cite{FirstPaper}.
  \label{asymptotic-rate-ALG}
  \end{theorem}
  \begin{proof}
  The proof of Theorem \ref{asymptotic-rate-ALF-Dk} go through until the analysis of $\mathbb{T}^2$-invariant closed anti-self-dual form $\phi$. Following the notations of our first paper, the basis of anti-self-dual forms can be written as
  $$\xi^1=\mathrm{d}u\wedge\mathrm{d}\bar{v},\xi^2=\mathrm{d}\bar{u}\wedge\mathrm{d}v,
  \xi^3=\mathrm{d}u\wedge\mathrm{d}\bar{u}-\mathrm{d}v\wedge\mathrm{d}\bar{v}$$
  When $(u,v)$ become $(e^{2\pi i\beta}u,e^{-2\pi i\beta}v)$,
  $(\xi^1,\xi^2,\xi^3)$ become $(e^{4\pi i\beta}\xi^1,e^{-4\pi i\beta}\xi^2,\xi^3).$
  Notice that $\phi$ can be decomposed into combinations of $u^{-\delta}\xi^1$, $\bar{u}^{-\delta}\xi^1$, $u^{-\delta}\xi^3$
  and their conjugates. Only the first one and its conjugate are closed.
  To make $u^{-\delta}\xi^1$ well defined, $-2\pi\beta\delta+4\pi\beta$ must be in $2\pi\mathbb{Z}$.
  Therefore, $\delta=\min_{n\in\mathbb{Z},n<2\beta}\frac{2\beta-n}{\beta}$.
  \end{proof}

  \begin{remark}
  In Theorem 1.5 of \cite{Hein}, Hein constructed lots of ALG gravitational instantons of order $\min_{n\in\mathbb{Z},n<2\beta}\frac{2\beta-n}{\beta}$ whose tangent cone at infinity has cone angle $2\pi\beta<2\pi$. Therefore, our estimate of asymptotic rate is optimal in ALG case if $\beta<1$.
  \end{remark}

  It's not hard to extend our method to ALH-non-splitting gravitational instanton using exponential growth weights and therefore complete the proof of Main Theorem 1. We will omit the details.

\section{Rigidity of multi-Taub-NUT metric}

 In this section, we analyze the ALF-$A_k$ gravitational instantons as a warm up of Main Theorem 2. We will use the twistor space method as in \cite{CherkisKapustin}. An important step in our approach is a compactification in the complex analytic sense and the analysis of topology of this compactification following Kodaira's work \cite{Kodaira}.

 We start from the compactification.

 \begin{theorem}
 Any ALF-$A_k$ gravitational instanton $(M,I)$ can be compactified in the complex analytic sense.
 \label{compactification-Ak}
 \end{theorem}

 \begin{proof}
 By the remark after Theorem \ref{asymptotic-rate-ALF-Dk}, $M$ is asymptotic to the standard model $E$ with error $O'(r^{-2})$. $E$ is either the trivial product $(\mathbb{R}^3\setminus B_R)\times\mathbb{S}^1$ or the quotient of the Taub-NUT metric outside a ball by $\mathbb{Z}_{k+1}$. In any case, there exist two $I$-holomorphic functions $z_E$ and $\rho_E$ satisfying $\omega^+=4i\mathrm{d}\log\rho_E\wedge\mathrm{d}z_E$. We are mostly interested in the behaviors when $x^1$ goes to $-\infty$. It corresponds to
 $$\mathbb{C}\times(\mathbb{C}^*\cap B_{e^{-R}})\cong\mathbb{C}\times(B_{e^{R}}^c)=\{(z_E,\rho_E):|\rho_E|>e^R\}.$$
 We are also interested in the corresponding part of $M$.

 On $M$, there exists an $I$-holomorphic function $z=z_E+O'(r^{-\delta})$ for any $\delta<1$. As in Section 4.7 of \cite{FirstPaper}, we can define a holomorphic vector field $X$ by $\omega^+(X,Y)=-i\mathrm{d}z(Y)$. On each fixed fiber, there exists a unique holomorphic form $\phi$ such that $\phi(X)=1$. Locally
 $$\omega^+=ic(z,v)\mathrm{d}z\wedge\mathrm{d}v, X=\frac{1}{c(z,v)}\frac{\partial}{\partial v}, \phi=c(z,v)\mathrm{d}v.$$

 Notice that each fiber in the interesting part of $M$ is topologically $\mathbb{C}^*\cap B_{e^{-R}}$. So on each fiber, we can integrate the form $\phi$ to get a holomorphic function $\chi \in \mathbb{C}/\mathbb{Z}\tau(u)$ up to a function of $z$. We can fix this ambiguity by requiring that $\chi-\Phi^*(-4\log \rho_E)$ goes to 0 when $\chi$ becomes negative infinity, where $\Phi$ is the map from $M$ to $E$. $\tau(u)=8\pi i$ since $M$ is asymptotic to $E$ and it's true on $E$. We can fix this ambiguity by writing $\chi$ as $\chi=-4\log \rho$. Therefore we get a part of $M$ biholomorphic to $\mathbb{C}\times(\mathbb{C}^*\cap B_{e^{-R}})\cong\{(z,\rho):|\rho|>e^{R}\}$ with $\omega^{+}=4i\mathrm{d}\log \rho\wedge\mathrm{d}z$. Similarly, the part of $M$ where $x^1$ goes to $+\infty$ is biholomorphic to $\mathbb{C}\times(\mathbb{C}^*\cap B_{e^{-R}})\cong\{(z,\xi):|\xi|>e^{R}\}$ with $\omega^{+}=4i\mathrm{d}z\wedge\mathrm{d}\log\xi$.

 Now we can add the divisors $D_-=\{\rho=\infty\}$ and $D_+=\{\xi=\infty\}$ to compactify the two parts. We can get a manifold with a holomorphic function $z$ whose generic fiber is $\mathbb{CP}^1$. Adding $D_\infty=\mathbb{CP}^1=\{z=\infty\}$, we can get a compact manifold $\bar M$ with a meromorphic function $z:\bar M\rightarrow \mathbb{CP}^1$ whose generic fiber is $\mathbb{CP}^1$.
 \end{proof}

 It's easy to see that $-K=\{\omega^+=\infty\}=D_-+D_+ +2D_\infty$ is the anti-canonical divisor. Any generic fiber is a non-singular rational curve $C=\mathbb{CP}^1$ with $(-KC)=2$ and $(C^2)=0$. Following the work of Kodaira \cite{Kodaira}, we can classify singular fibers.

 \begin{theorem}
 Any singular fiber $C$ can be written as the sum of non-singular rational curves
 $$C=\Theta_0+...+\Theta_m,m=1,2,3,...,$$ with $$(\Theta_i\Theta_j)=\delta(|j-i|-1),$$ $$(\Theta_i^2)=-2+\delta(0)+\delta(m),$$
 $$(-K\Theta_i)=\delta(0)+\delta(m),$$
 where $\delta(n)=1$ if $n=0$, and $\delta(n)=0$ otherwise.
 \label{singular-fiber-Ak}
 \end{theorem}

 \begin{proof}
 Let $C=\sum n_i\Theta_i$. The main tools are Kodaira's identities \cite{Kodaira}
 $$2\pi'(\Theta_i)-2-(\Theta_i^2)=(K\Theta_i),$$
 $$(C\Theta_i)=0=n_i(\Theta_i^2)+\sum_{j\not=i}n_j(\Theta_i\Theta_j),$$
 where the virtual genus $\pi'(\Theta_i)$ is non-negative and $\pi'(\Theta_i)$ vanishes if and only if $\Theta_i$ is a non-singular rational curve.

 If there is only one curve $C=\Theta_0$, then $(K\Theta_0)=-2$. Notice that $(\Theta_0^2)=0$ by the second identity. So $\pi'(\Theta_0)=0$ by the first identity. So $\Theta_0$ is a non-singular rational curve. In other words, the fiber is regular.

 Otherwise, from the information near $D_\pm$, there exist two curves $\Theta_0$ and $\Theta_m$ satisfying
 $(-K\Theta_0)=(-K\Theta_m)=1$ and $n_0=n_m=1$. All other curves don't intersect $-K$.
 From the second identity and the fact that $C$ is connected \cite{Kodaira}, we know that $(\Theta_i^2)<0$. Therefore, $\pi'(\Theta_i)$ must be 0, i.e. each $\Theta_i$ is a non-singular rational curve.
 It follows that $(\Theta_i^2)=-2+\delta(0)+\delta(m)$.

 Now the second identity becomes $1=\sum_{j\not=0}n_j(\Theta_0\Theta_j)$.
 If $(\Theta_0\Theta_m)=1$, we are done with $m=1$.
 Otherwise, suppose $(\Theta_0\Theta_1)=1$.
 Then $n_1=1$, so we get $2=1+\sum_{j\not=0,j\not=1}n_j(\Theta_1\Theta_j)$.
 If $(\Theta_1\Theta_m)=1$, we are done with $m=2$.
 Otherwise, we can continue.
 After several steps, we must stop because the number of curves is finite. Therefore, the singular fibers must have the required properties.
 \end{proof}

 \begin{remark}
 That's exactly the picture in Theorem \ref{LeBrun}.
 \end{remark}

 \begin{remark}
 If each fiber is regular except the fiber $\{z=0\}$, $M$ is biholomorphic to the minimal resolution of $xy=z^{k+1}$. In this case, the central fiber has $k=m-1$ non-singular rational curves $\Theta_1$,...,$\Theta_{m-1}$ whose intersection diagram is called the $A_k$ Dynkin diagram. That's the reason why we call $M$ ALF-$A_k$.
 \end{remark}

 Now we are able to give a new proof of the following theorem. It was first proved by Minerbe in \cite{MinerbeMultiTaubNUT} using the existence of Killing vector fields. However, since there is no Killing vector field on ALF-$D_k$ gravitational instantons, we prefer a new proof of this theorem using the twistor space.

 \begin{theorem}
 (Minerbe)
 Any ALF-$A_k$ gravitational instanton must be the multi-Taub-NUT metric.
 \end{theorem}

 \begin{proof}
 First of all, let's look at the slice $\zeta=0$. In other words, we use $I$ as the complex structure. $\omega^+$ as the holomorphic symplectic form. By Theorem \ref{compactification-Ak}, there exist $\rho$ and $\xi$ such that $\omega^{+}=4i\mathrm{d}\log\rho\wedge\mathrm{d}z=4i\mathrm{d}z\wedge\mathrm{d}\log\xi$. So $\rho\xi$ is a holomorphic function of $z$ satisfying $\lim_{z\rightarrow \infty}\rho\xi/z^{k+1}=1$. It's completely determined by its zeros. By Theorem \ref{singular-fiber-Ak} and Theorem \ref{LeBrun}, it's easy to see that $(M,I)$ is biholomorphic to $\rho\xi=\prod_{\alpha=1}^{k+1}(z-P_\alpha)$ or the minimal resolution of it.

 Now we may vary $\zeta\not=\infty$. We can still get
 $\rho\xi=\prod_{\alpha=1}^{k+1}(z-P_\alpha(\zeta))$ with $\omega=4i\mathrm{d}\log\rho\wedge\mathrm{d}z$.
 Similarly, for $\zeta\not=0$, we may use $\tilde\zeta=\zeta^{-1}$ instead. Then $\tilde \omega=\zeta^{-2}\omega$ and $\tilde z=\zeta^{-2}z$ are non-singular. So we can get $\tilde\omega=4i\mathrm{d}\log\tilde\rho\wedge\mathrm{d}\tilde z$ instead.
 The difference $\tilde \rho/\rho$ is a holomorphic function of $\zeta$ and $z$.
 It equals to $e^{-z/\zeta}\zeta^{-k-1}$ on $E$, so $\tilde \rho/\rho$ must be $e^{-z/\zeta}\zeta^{-k-1}$ on $M$. Similarly, $\tilde\xi=e^{z/\zeta}\zeta^{-k-1}\xi$.

 Since $\tilde{\rho}\tilde{\xi}=\prod_{\alpha=1}^{k+1}(\tilde{z}-\tilde{P}_\alpha(\tilde\zeta))$.
 It's easy to see that $P_\alpha(\zeta)=\zeta^2\tilde P_{\alpha}(\tilde \zeta)$.
 So $P_\alpha(\zeta)$ must be a degree two polynomial of $\zeta$.

 Now let's look at the action of the real structure. When $\zeta$ becomes $-1/\bar \zeta$, $M$ becomes exactly its own conjugation. Since $z$ is invariant under the action $(\zeta,z)\rightarrow(-1/\bar \zeta,-\bar z/\bar\zeta^{2})$, $P_\alpha$ must have the same property under the real structure.
 In other words, $P_{\alpha}(\zeta)=a_{\alpha}\zeta^2+2b_{\alpha}\zeta-\bar a_{\alpha}$ for some $a_{\alpha}\in\mathbb{C}$ and $b_{\alpha}\in\mathbb{R}$.
 It's easy to see that the real structure $\tau$ must act by $$\tau(\zeta,z,\rho,\xi)=(-1/\bar\zeta,-\bar z/\bar\zeta^{2},e^{\bar z/\bar\zeta}(1/\bar \zeta)^{k+1}\bar\xi,e^{-\bar z/\bar\zeta}(-1/\bar\zeta)^{k+1}\bar\rho).$$

 It's well known \cite{HithcinKarlhedLindstromRocek} that the form $\omega$ and the real structure on the twistor space determine the metric on $M$. So $M$ must be the multi-Taub-NUT metric.
 \end{proof}

 \section{Classification of ALF-$D_k$ gravitational instantons}
 In this section we prove Main Theorem 2 as we did for the ALF-$A_k$ gravitational instantons in the previous section.

 We still start from the compactification.

 \begin{theorem}
 Any ALF-$D_k$ gravitational instanton $(M,I)$ can be compactified in the complex analytic sense.
 \label{compactification-Dk}
 \end{theorem}

 \begin{proof}
 We already know that outside a compact set, $M$ is up to $O'(r^{-3})$, the $\mathbb{Z}_2$-quotient of a standard $\mathbb{S}^1$-fiberation $E$ over $\mathbb{R}^3-B_R$. Moreover, there is an $I_M$-holomorphic function $z=(-x^3+ix^2)^2+O'(r^{-1})$ on $M$.

 Recall that there is a part of $(E,I_E)$ biholomorphic to $$\mathbb{C}\times(\mathbb{C}^*\cap B_{e^{-R}})=\{(a_E,b_E):a_E=-x^3+ix^2\in\mathbb{C},b_E\in\mathbb{C}^*\cap B_{e^{-R}}\}.$$ Now we claim that the corresponding part of $(M,I_M)$ is also biholomorphic to $$\mathbb{C}\times(\mathbb{C}^*\cap B_{e^{-R}})=\{(a,b):a\in\mathbb{C},b\in\mathbb{C}^*\cap B_{e^{-R}}\}$$
 What's more, under this diffeomorphism, $\omega^{+}_M=-4i\mathrm{d}\log b\wedge\mathrm{d}a$.

 It's hard to solve $a$, $b$ as functions of $a_E$, $b_E$ directly. However, following the idea of Newlander and Nirenberg \cite{NewlanderNirenberg}, we can instead solve $a_E$, $b_E$ as functions of $a$, $b$ and apply the inverse function theorem.

 Let $$a=u+iv, \log b=t+i\theta, a_E=a+z^1, b_E=be^{z^2}.$$
 Let $$\partial_1=\frac{\partial}{\partial a}=\frac{1}{2}(\frac{\partial}{\partial u}-i\frac{\partial}{\partial v}),
 \bar\partial_1=\frac{\partial}{\partial \bar a},
 \partial_2=b\frac{\partial}{\partial b}=\frac{1}{2}(\frac{\partial}{\partial t}-i\frac{\partial}{\partial \theta}),
 \bar\partial_2=\bar b\frac{\partial}{\partial \bar b}.$$
 Then the equation is reduced to
 $$\bar\partial_j z^k+\phi^{k}_l(u+iv+z^1,t+i\theta+z^2)(\bar\partial_j\bar z^{l}+\delta_{j}^{l})=0,$$
 where $$|\nabla^m \phi(u+iv,t+i\theta)|<C(m)(u^2+v^2+t^2)^{(-3-m)/2}$$
 for all $m\ge 0$ and all $t<-R$ if $R$ is large enough.

 Instead of the space $\tilde C^{n+\alpha}$ in \cite{NewlanderNirenberg}, we prefer weighted Hilbert spaces.

 Define $$||f||_{L^2_{\alpha,\beta}}=\int_{t<-R}|f|^2(1+u^2+v^2)^{\alpha/2}|t|^{\beta},$$
 and $$||f||_{H^{m}_{\alpha,\beta}}=\sqrt{\sum_{i+j+k+l\le m}||\partial_u^i\partial_v^j\partial_t^k\partial_\theta^l f||^2_{L^2_{\alpha+2i+2j,\beta+2k+2l}}},$$
 then we can find an operator $$T_1:L^2_{\alpha,\beta}\rightarrow L^2_{\alpha-2,\beta}$$
 satisfying $$\bar\partial_1T_1f=f$$
 in the distribution sense if $\alpha<2$ and $\alpha$ isn't an integer. Actually, by Theorem 4.12 of our first paper \cite{FirstPaper}, we can find $G_1$ such that
 $4\partial_1\bar\partial_1G_1f=f$ in the distribution sense. So $T_1f=4\partial_1G_1f$.

 Similarly, by Theorem \ref{full-image}, we can find an operator $$T_2:L^2_{\alpha,\beta}\rightarrow L^2_{\alpha,\beta-2}$$
 satisfying $$\bar\partial_2T_2f=f$$ in the distribution sense if $\beta$ isn't an integer. Since both $T_1$ and $T_2$ are canonically defined, $T_1$ commutes with $\partial_2$ and $\bar\partial_2$ while $T_2$ commutes with $\partial_1$ and $\bar\partial_1$.
 By the work of Newlander and Nirenberg \cite{NewlanderNirenberg}, the integrability condition implies that it's enough to solve the equation
 $$z^i=T^1f^i_1+T^2f^i_2-\frac{1}{2}T^1\bar\partial_1T^2f^i_2-\frac{1}{2}T^2\bar\partial_2T^1f^i_1,$$
 where $$f^i_j=-\phi^i_l(u+iv+z^1,t+i\theta+z^2)(\bar\partial_j\bar z^{l}+\delta_{j}^{l}).$$
 It has a unique solution in $H^{10}_{-2\epsilon,-2\epsilon}$ for any $0<\epsilon<1/2$ if $R$ is large enough.
 By Sobolev embedding theorem, $|z^i|\le C(1+u^2+v^2)^{(-1+\epsilon)/2}|t|^{-1+\epsilon}$.

 In conclusion, we've solved $a_E$ and $b_E$ in terms of $a$ and $b$. We can invert them to get $a$ and $b$ in terms of $a_E$ and $b_E$. By the arguments similar to Theorem \ref{compactification-Ak}, we can slightly modify $b$ such that $\lim_{b\rightarrow 0}(b/b_E)=1$ and $\omega^{+}_M=-4i\mathrm{d}\log b\wedge\mathrm{d}a$.

 Therefore, we can add the divisor $D=\{b=0\}$ to compactify this part. On $M\cup D$, the condition $z=a_E^2+O'(r^{-1})$ is reduced to $z(a,0)=a^2$. Near $a=\infty$, let $c=1/a$, then $M\cup D$ is locally biholomorphic to $$((\mathbb{C}^*\cap B_{1/R})\times\mathbb{CP}^1)/\mathbb{Z}_2=\{(c,b):0<|c|<1/R,b\in\mathbb{CP}^1\}/(c,b)\sim(-c,1/b).$$
 As Kodaira did in \cite{Kodaira}, we can add $\{(0,b)\}/(0,b)\sim(0,1/b)$ and then replace the neighborhoods of two singular points $(0,1)$ and $(0,-1)$ by two copies of $N_{+2}$.(See page 583 of \cite{Kodaira}). Similar to page 586 of \cite{Kodaira}, we know that $D_\infty=2\Theta+\Theta_0+\Theta_1$ with $(\Theta\Theta_i)=(D\Theta)=1$, $(\Theta_1\Theta_2)=(D\Theta_i)=0$, $(\Theta^2)=-1$ and $(\Theta_i^2)=-2$. Therefore, we get a compact manifold $\bar M=M\cup D\cup D_\infty$ with a meromorphic function $z:\bar M\rightarrow \mathbb{CP}^1$ whose generic fiber is a non-singular rational curve.
 \end{proof}

 On $\bar M$, the anti-canonical divisor $-K=\{\omega^+=\infty\}=D+D_\infty$. Any generic fiber is a non-singular rational curve $C=\mathbb{CP}^1$ with $(-KC)=2$ and $(C^2)=0$. Any singular fiber $\{z=z_0\}$ must belong to the list in Theorem \ref{singular-fiber-Ak} if $z_0\not=0, \infty$. So we only need to classify the fiber $\{z=0\}$. The main property is that $-K=D+D_\infty$ intersects $C$ at only one point.

 \begin{theorem}
 The fiber $C=\{z=0\}$ can be written as the sum of non-singular rational curves. There are three cases:

 (1) $C=\Theta$, $(\Theta^2)=0$, $(-K\Theta)=(D\Theta)=2$, but $D$ intersects $\Theta$ at one point with multiplicity 2.

 (2) $C=\Theta_0+\Theta_1$, $(\Theta_0^2)=(\Theta_1^2)=-1$, three curves $\Theta_0$, $\Theta_1$, $D$ intersect at same point.

 (3) $C=2\Theta_0+...+2\Theta_m+\Theta_{m+1}+\Theta_{m+2}$, $m=0,1,...$

 $$(\Theta_0\Theta_1)=...=(\Theta_{m-1}\Theta_m)=(\Theta_m\Theta_{m+1})=(\Theta_m\Theta_{m+2})=1,$$
 $$(\Theta_0^2)=-1, (\Theta_1^2)=...=(\Theta_{m+2}^2)=-2, (-K\Theta_0)=(D\Theta_0)=1,$$
 and all other intersection numbers are 0.
  \label{singular-fiber-Dk}
 \end{theorem}

 \begin{proof}
 Let $C=\sum n_i\Theta_i$. We still use Kodaira's two identities
 $$2\pi'(\Theta_i)-2-(\Theta_i^2)=(K\Theta_i),$$
 $$(C\Theta_i)=0=n_i(\Theta_i^2)+\sum_{j\not=i}n_j(\Theta_i\Theta_j)$$
 and the fact the $C$ is connected. By the second identity,
 $$(\Theta_i^2)=-\frac{1}{n_i}\sum_{j\not=i}n_j(\Theta_i\Theta_j)\le 0.$$

 Since $(DC)=(-KC)=2$, but $D$ intersects $C$ at only one point, there are only three possibilities.

 (1) $\Theta_0$ intersects $D$ at one point with multiplicity 2. By Kodaira's first identity, $(\Theta_0^2)=0$ and $\pi'(\Theta_0)=0$. Therefore, there are no other curves at all. It's the first case.

 (2) $\Theta_0$ and $\Theta_1$ intersect $D$ at same point. So $(\Theta_i^2)=-1$ and $\pi'(\Theta_i)=0$. There are still no other curves at all. It's the second case.

 (3) $\Theta_0$ intersects $D$ at one point but $n_0=2$. In this case, $(\Theta_0^2)=-1$ and $\pi'(\Theta_0)=0$ by Kodaira's first identity. As in Theorem \ref{singular-fiber-Ak}, since any other curve has no intersection with $D$, it must be a non-singular rational curve with self intersection number $-2$.

 Therefore, either two different curves $\Theta_1$ and $\Theta_2$ intersect $\Theta_0$ or one curve $\Theta_1$ intersects $\Theta_0$ but $n_1=2$. In the first case, we are done. In the second case, we can continue the same kind of analysis. After finite steps, we are done since there are only finitely many curves.
 \end{proof}

 \begin{remark}
 If each fiber is regular except the fiber $\{z=0\}$, $M$ is biholomorphic to the minimal resolution of $x^2-zy^2=-z^{k-1}$. In this case, the central fiber has $k=m+2$ non-singular rational curves $\Theta_1$,...,$\Theta_{m+2}$ whose intersection diagram is called the $D_k$ Dynkin diagram. That's the reason why we call $M$ ALF-$D_k$.
 \end{remark}

 Now, we are able to prove Main Theorem 2.

 \begin{theorem}
 (Main Theorem 2) Any ALF-$D_k$ gravitational instanton must be the Cherkis-Hitchin-Ivanov-Kapustin-Lindstr\"om-Ro\v{c}ek metric.
 \end{theorem}

 \begin{proof}
 We still start from the slice $\zeta=0$. We already know that the double cover $\tilde M$ of $M\setminus\{z=0\}$ is asymptotic to $E$. $\sqrt{z}\approx a$ is well defined on $\tilde M$. As before, we can define a holomorphic function $f$ on $\tilde M$ by $\omega^{+}=4i\mathrm{d}\log f\wedge\mathrm{d}\sqrt{z}$ and $\lim_{b\rightarrow0}fb=1$. The composition of $f$ and the covering transform is called $f'$.

 Now we are interested in the behavior near $a=0$ and $b=0$. We can write $z$ as
 $$z=e^{bf_1(a,b)}(a^2+bf_2(b)a+bf_3(b)).$$
 The function $f_4(b)=-(bf_2(b))^2/4+bf_3(b)$ can be written as $cb^{m}(1+bf_5(b))$, where $m=1,2,\infty$ depending on the type of $\{z=0\}$.
 Change the coordinates by
 $$a'=e^{bf_1(a,b)/2}(a+bf_2(b)/2),$$
 $$b'=be^{bf_1(a,b)/m}(1+bf_5(b))^{1/m}.$$
 Then $$z=a'^2+cb'^m,$$
 with $$\omega^+=-4i\mathrm{d}\log b'\wedge(1+b'f_6(a,b))\mathrm{d}a'=4i\mathrm{d}\log f\wedge\mathrm{d}\sqrt{z}.$$

 (1) In the first case of Theorem \ref{singular-fiber-Dk}, $m=1$.
 So $$\log f=-\int\frac{\sqrt{z}\mathrm{d}b'}{a'b'}=-\log b'+\log\frac{(\sqrt{z}+a')^2}{4z}=\log\frac{a'+\sqrt{z}}{a'-\sqrt{z}}+\log\frac{-c}{4z}$$
 if we ignore the term $b'f_6(a,b)$.
 However, the contribution from the term $b'f_6(a,b)$ is bounded by $C\int_0^{e^{-R}}|\frac{\sqrt{|z|}}{\sqrt{||z|-|c|t|}}|\mathrm{d}t\le C\sqrt{|z|}$.
 So $\lim_{z\rightarrow 0}fz=-c/4$. It's also true that $\lim_{z\rightarrow 0}f'z=-c/4$. Therefore, we can write $f, f'$ as
 $$fz=P+\sqrt{z}Q, f'z=P-\sqrt{z}Q.$$
 Away from $\{z=0\}$, the picture is similar to the $A_{2k-5}$ case, so
 $$P^2-zQ^2=(P+\sqrt{z}Q)(P-\sqrt{z}Q)=\prod_{\alpha=1}^{k}(z-P_{\alpha}^2).$$
 Notice that $\lim_{z\rightarrow 0}P=-c/4=\prod_{\alpha}(-iP_{\alpha})$, so we can write $P=yz+\prod_{\alpha}(-iP_{\alpha})$ and write $Q=x$.
 A simple calculation yields
 $$x^2-zy^2=\frac{1}{-z}(\prod_{\alpha}(z-P^2_{\alpha})-\prod_{\alpha}(-P^2_{\alpha}))+2\prod_{\alpha}(-iP_{\alpha})y$$ and
 $$\omega^+=i\mathrm{d}(\frac{1}{\sqrt{z}}\log(\frac{P+\sqrt{z}Q}{P-\sqrt{z}Q}))\wedge\mathrm{d}z
 =i\mathrm{d}(\frac{1}{\sqrt{z}}\log(\frac{yz+\prod_{\alpha}(-iP_{\alpha})
  +\sqrt{z}x}{yz+\prod_{\alpha}(-iP_{\alpha})-\sqrt{z}x}))\wedge\mathrm{d}z.$$

 (2) In the second case of Theorem \ref{singular-fiber-Dk}, $m=2$.
 So $$\log f=-\int\frac{\sqrt{z}\mathrm{d}b'}{a'b'}=
 -\log b'+\log\frac{\sqrt{z}+a'}{2\sqrt{z}}=\frac{1}{2}\log\frac{a'+\sqrt{z}}{a'-\sqrt{z}}+\frac{1}{2}\log\frac{-c}{4z}$$
 if we ignore the term $b'f_6(a,b)$.
 In this case, the contribution from the term $b'f_6(a,b)$ is bounded by $C\int_0^{e^{-R}}|\frac{\sqrt{|z|}}{\sqrt{||z|-|c|t^2|}}|\mathrm{d}t\le C\sqrt{|z|}\log(1/\sqrt{|z|})$. So $\lim_{z\rightarrow 0}f\sqrt{z}=\sqrt{-c}/2.$ It's also true that $\lim_{z\rightarrow 0}-f'\sqrt{z}=\sqrt{-c}/2.$
 So we can write $f, f'$ as
 $$f\sqrt{z}=x+\sqrt{z}y, -f'\sqrt{z}=x-\sqrt{z}y.$$
 It's easy to see that $$x^2-zy^2=(x+\sqrt{z}y)(x-\sqrt{z}y)=-\prod_{\alpha=1}^{k-1}(z-P_{\alpha}^2).$$
 Notice that $\lim_{z\rightarrow 0}x=\sqrt{-c}/2$ on $\Theta_0$, but $\lim_{z\rightarrow 0}x=-\sqrt{-c}/2$ on $\Theta_1$,
 so we can no longer reduce $x$ and $y$. However, let $P_{k}=0$, then
 $$x^2-zy^2=\frac{1}{-z}(\prod_{\alpha=1}^{k}(z-P^2_{\alpha})-\prod_{\alpha=1}^{k}(-P^2_{\alpha}))+2\prod_{\alpha=1}^{k}(-iP_{\alpha})y$$ and
 $$\omega^+=4i\mathrm{d}\log f\wedge\mathrm{d}\sqrt{z}
 =i\mathrm{d}(\frac{1}{\sqrt{z}}\log(\frac{yz+\prod_{\alpha=1}^{k}(-iP_{\alpha})
  +\sqrt{z}x}{yz+\prod_{\alpha=1}^{k}(-iP_{\alpha})-\sqrt{z}x}))\wedge\mathrm{d}z.$$
 It's convenient to write $P=yz$ and $Q=x$. So $fz=P+\sqrt{z}Q$ still holds.

 (3) In the third case of Theorem \ref{singular-fiber-Dk}, $m=\infty$.
 Just as we did in Theorem \ref{compactification-Dk} near $z=\infty$, the manifold becomes the minimal resolution of the $\mathbb{Z}_2$-quotient of multi-Taub-NUT metric. $\mathbb{Z}_2$ acts by interchanging $\rho$ and $\xi$. So $f=\rho$ and $f'=\xi$. They satisfy $ff'=\prod_{\alpha=1}^{k-2}(z-P_{\alpha}^2)$. Let $x=\sqrt{z}(f-f')/2$ and $y=(f+f')/2$.
 Then $x^2-zy^2=-z\prod_{\alpha=1}^{k-2}(z-P_{\alpha}^2)$. Let $P_{k-1}=P_{k}=0$. Then $(M,I)$ is biholomorphic to the minimal resolution of
 $$x^2-zy^2=\frac{1}{-z}(\prod_{\alpha=1}^{k}(z-P^2_{\alpha})-\prod_{\alpha=1}^{k}(-P^2_{\alpha}))+2\prod_{\alpha=1}^{k}(-iP_{\alpha})y$$
 and
 $$\omega^+=4i\mathrm{d}\log f\wedge\mathrm{d}\sqrt{z}
 =i\mathrm{d}(\frac{1}{\sqrt{z}}\log(\frac{yz+\prod_{\alpha=1}^{k}(-iP_{\alpha})
  +\sqrt{z}x}{yz+\prod_{\alpha=1}^{k}(-iP_{\alpha})-\sqrt{z}x}))\wedge\mathrm{d}z.$$
 Let $P=yz$ and $Q=x$. Then $fz=P+\sqrt{z}Q$ still holds.

 In conclusion, we always have the correct biholomorphic type and correct $\omega^+$. The only difference is how many $P_{\alpha}$'s equal to 0.
 Now we may vary $\zeta\not=\infty$. We can still get similar pictures.
 For $\zeta\not=0$, we may use $\tilde\zeta=\zeta^{-1}$ instead. Then $\tilde \omega=\zeta^{-2}\omega$ and $\tilde z=\zeta^{-4}z$ are non-singular. So we can get $\tilde f$, $\tilde f'$ $\tilde x$, $\tilde y$, $\tilde P$, and $\tilde Q$ instead.
 The difference $\tilde f/f$ transfer as $\tilde \rho/\rho$ in the $A_{2k-5}$ case. Therefore $\tilde f/f=e^{-\sqrt{z}/\zeta}\zeta^{-2k+4}$.
 So $(\tilde P+\sqrt{\tilde z}\tilde Q))/(P+\sqrt{z}Q)=e^{-\sqrt{z}/\zeta}\zeta^{-2k}$.

 It's conventional to rescale the metric. Therefore, we actually have $$(\tilde P+\sqrt{\tilde z}\tilde Q))/(P+\sqrt{z}Q)=e^{-2\sqrt{z}/\zeta}\zeta^{-2k}$$ as our transition function instead.
 In other words,

 $$ \left( {\begin{array}{*{20}c}
    \tilde P  \\
    \tilde Q   \\
 \end{array}} \right) = \zeta^{-2k}\left( {\begin{array}{*{20}c}
    \cosh(2\sqrt{z}/\zeta) &  -\sqrt{z}\sinh(2\sqrt{z}/\zeta) \\
    -\zeta^2\sinh(2\sqrt{z}/\zeta)/\sqrt{z} &  \zeta^2\cosh(2\sqrt{z}/\zeta)  \\
 \end{array}} \right)\left( {\begin{array}{*{20}c}
    P  \\
    Q  \\
 \end{array}} \right).
 $$

 As before, $P_\alpha(\zeta)$ must be a degree two polynomial in $\zeta$. When we look at the action of the real structure, it's easy to see that actually $P_{\alpha}(\zeta)=a_{\alpha}\zeta^2+2b_{\alpha}\zeta-\bar a_{\alpha}$ for some $a_{\alpha}\in\mathbb{C}$ and $b_{\alpha}\in\mathbb{R}$. Moreover, the real structure $\tau$ must act by $$\tau(\zeta,z,P,Q)=(\tilde\zeta=-\bar\zeta,\tilde z=\bar z,\tilde P=\bar P,\tilde Q=-\bar Q).$$

 We can further transfer those expressions into $x$ and $y$ by the fact that $P=yz+\prod_{\alpha}(-iP_\alpha(\zeta))$, $Q=x$ and
 $\tilde P=\tilde y\tilde z+\prod_{\alpha}(-i\tilde P_\alpha(\tilde \zeta))$, $\tilde Q=\tilde x$.

 It's well known \cite{HithcinKarlhedLindstromRocek} that the form $\omega$ and the real structure on the twistor space determine the metric on $M$. So $M$ must be the Cherkis-Hitchin-Ivanov-Kapustin-Lindstr\"om-Ro\v{c}ek metric.
 \end{proof}

 \section{A Torelli-type theorem for ALF gravitational instantons}
 In this section we prove the Torelli-type theorem for ALF gravitational instantons as an analogy of Kronheimer's results \cite{Kronheimer1} \cite{Kronheimer2}.

 First of all, we can rescale the metric to make the scaling parameter $\mu=1$.

 In the ALF-$A_k$ case, for each $\alpha\not=\beta$, $\pi^{-1}$ of the segment connecting $\mathbf{x}_\alpha$ and $\mathbf{x}_\beta$ is a sphere $S_{\beta,-\alpha}$. They generate $H_2(M,\mathbb{Z})$. It's easy to see that they are the only roots, i.e. homology classes with self-intersection number -2. The simple roots can be chosen as $S_{2,-1},S_{3,-2},...,S_{k+1,-k}$. They form an $A_k$ root system.

 By a simple calculation, $$\int_{S_{\beta,-\alpha}}\omega=\int_{S_{\beta,-\alpha}}4i\mathrm{d}\log\rho\wedge\mathrm{d}z
 =8\pi\int_{\pi(S_{\beta,-\alpha})}\mathrm{d}z=8\pi(P_\beta-P_\alpha).$$
 So $(\int_{S_{\beta,-\alpha}}\omega^1,\int_{S_{\beta,-\alpha}}\omega^2,\int_{S_{\beta,-\alpha}}\omega^3)$ equals to
 $(b_\beta-b_\alpha,\mathrm{Re}(a_\beta-a_\alpha),\mathrm{Im}(a_\beta-a_\alpha))$ up to a constant multiple.
 Since the ALF hyperk\"ahler structure is completely determined by the parameters $(a_\beta-a_\alpha,b_\beta-b_\alpha)$, it's also determined by three cohomology classes $[\omega^i]$.

 The ALF-$A_k$ gravitational instanton is singular if and only if there exist $\alpha\not=\beta$ such that $(a_\alpha,b_\alpha)=(a_\beta,b_\beta)$. It's equivalent to the vanishing of $[\omega^i]$ on some root.

 The ALF-$D_k$ case is similar. When $k\ge 2$, the roots $S_{\pm\beta,\pm\alpha}, \alpha\not=\beta$ generate $H_2(M,\mathbb{Z})$. The simple roots can be chosen as $S_{+2,+1},S_{+2,-1},S_{+3,-2}$,
 $S_{+4,-3}...,S_{+k,-(k-1)}$. They form a $D_k$ root system. The integrals on them $(\int_{S_{\pm\beta,\pm\alpha}}\omega^1,\int_{S_{\pm\beta,\pm\alpha}}\omega^2,\int_{S_{\pm\beta,\pm\alpha}}\omega^3)$ are $(\pm b_\beta\pm b_\alpha,\mathrm{Re}(\pm a_\beta\pm a_\alpha),\mathrm{Im}(\pm a_\beta\pm a_\alpha))$ up to a constant multiple, too. The ALF-$D_k$ gravitational instanton is singular if and only if there exist $\alpha\not=\beta$ such that $(a_\alpha,b_\alpha)=\pm(a_\beta,b_\beta)$. So the Torelli-type theorem also holds in this case.

 When $k=1$, $H_2(M,\mathbb{Z})$ is generated by $S_{+1,-1}$, a sphere with one ordinary double point. Its self-intersection number is 0. It's easy to see that the Torelli-type theorem holds, too.

 When $k=0$, $H_2(M,\mathbb{Z})=0$ and there is only one ALF-$D_0$ gravitational instanton. The Torelli-type theorem holds trivially.

 \section{Applications}
 There are lots of different ways to construct ALF gravitational instantons. Our work shows that they are essentially the same thing. This gives us the relationship among different constructions.

 The original idea of Ivanov-Lindstrom-Ro\v{c}ek \cite{LindstromRocek} \cite{IvanovRocek} comes from the supersymmetric sigma model. Later, their conjecture was realized by Cherkis-Hithcin-Kapustin \cite{CherkisKapustin} \cite{CherkisHitchin} as the moduli space of magnetic monopoles, i.e. solutions of the Bogomolny equation with prescribed singularities. The ALF-$D_k$ gravitational instanton also appears in the supersymmetric N=4 SU(2) gauge theories \cite{SeibergWitten}, string theory and M-theory \cite{Sen}.

 An anti-self-dual harmonic form on the ALF-$D_0$ gravitational instanton was computed by Sen \cite{SenSdualityI} \cite{SenSdualityII} as the evidence of the S-duality, a generalization of the electric-magnetic duality. The cohomology groups of more general gravitational instantons were computed by Hausel-Hunsicker-Mazzeo \cite{HauselHunsickerMazzeo}.

 The Yang-Mills instantons on ALE gravitational instantons were studied by Kronheimer-Nakajima \cite{KronheimerNakajima} as the generalization of ADHM construction \cite{AtiyahDrinfeldHitchinManin}. It's closely related to representations of quivers and Kac-Moody algebras \cite{Nakajima}. Cherkis \cite{Cherkis} generalized this to representations of bows in order to study the instantons on ALF gravitational instantons. He said, from the string theory picture of \cite{Witten}, it is more natural to consider instantons on ALF, rather than on ALE spaces. The ALF gravitational instantons can also be realized as moduli spaces of certain bow representations.

 We've seen the relationship between gravitational instantons and the representation theory. The Weierstrass elliptic function (page 723 of \cite{CherkisKapustin}) and the modular forms (page 20 of \cite{Ionas}) also appear in the calculation of the ALF gravitational instantons. It's not just a coincide. In fact, according to Kapustin-Witten \cite{KapustinWitten}, there is a deep relationship between the geometric Langlands program and the ALF-$D_k$ gravitational instantons (Section 9 of \cite{KapustinWitten}), S-duality, N=4 super Yang-Mills Theory and the Hitchin fiberation, i.e. the map from the moduli space of Higgs bundles to characteristic polynomials. Thus, we hope that our work can shed light on the geometric Langlands program and eventually, on the Langlands program.

 A related remarkable progress was made by Laumon and Ng\^o \cite{LaumonNgo}. They proved the fundemental lemma of the Langlands program for unitary groups using a finite field analogy of the Hitchin fibration introduced by Ng\^o \cite{Ngo}. According to Boalch \cite{Boalch}, the simplest Hitchin fibration is provided by the elliptic fibration of the ALG gravitational instanton. We will leave the ALG gravitational instanton for further study after the improvement of the asymptotic rate in Main Theorem 1.

 Compared to the applications to other areas of mathematics and physics, the application to the differential geometry is rather limited. One may image that the gravitational instanton arises as the bubble of some geometric constructions. It's easy to show that a bubble has finite energy. By the work of Cheeger-Tian \cite{CheegerTian}, the curvature must decay quadratically. However, according to Gromov (Page 96 of \cite{Gromov}), the quadratic curvature decay condition can't provide even the weakest topological information. On the contrary, the faster than quadratic curvature decay condition contains much more information but is hard to obtain. Nevertheless, when we reverse the process, our Main Theorem 1 provides an improved asymptotic rate for the gluing constructions of lots of interesting geometric structures.

 \end{document}